\documentclass[12pt]{article}

\usepackage{amscd,amsmath,amssymb,amsthm,amsfonts,epsfig,graphics}
\usepackage{fullpage}
\usepackage{appendix}
\usepackage{graphicx}
\usepackage{amsmath}
\usepackage{amsthm}
\usepackage[byname]{smartref}
\usepackage{hyperref} 
\usepackage{tocloft}
\usepackage{tikz}
\usepackage{xcolor}
\usepackage{listings}
\usetikzlibrary{positioning}
\usepackage{tikz-qtree,tikz-qtree-compat}
\usetikzlibrary{trees}
\usetikzlibrary{arrows,positioning,automata,shadows,fit,shapes}
\usepackage{mathtools}
\usepackage{float}
\usepackage{amssymb}
\usepackage{tikz-cd}
\usepackage[english]{babel}
\usepackage{authblk}
\usepackage{tikz-cd}
\usepackage{graphicx}
\usepackage{amsmath}
\usepackage{amsfonts}
\usepackage{mathrsfs}
\usepackage{bm}

\theoremstyle{plain}
\newtheorem{theorem}{Theorem}
\newtheorem{lemma}[theorem]{Lemma}

\newtheorem{definition}[theorem]{Definition}

\newtheorem{example}[theorem]{Example}
\newtheorem*{remark}{Remark}

\newcommand{\F}{\ensuremath{\mathbb{F}}}
\newcommand{\Z}{\ensuremath{\mathbb{Z}}}
\newcommand{\Q}{\ensuremath{\mathbb{Q}}}

\DeclareBoldMathCommand{\be}{e}
\DeclareBoldMathCommand{\bx}{x}
\DeclareBoldMathCommand{\by}{y}

\title{Algebraic Construction of Quasi-split Algebraic Tori}
\author[1]{Armin Jamshidpey}
\author[2]{Nicole Lemire}
\author[1]{\'Eric Schost}
\affil[1]{David Cheriton School of Computer Science, University of Waterloo}
\affil[2]{Department of Mathematics, University of Western Ontario}

\begin{document}

\maketitle

\begin{abstract}

The main purpose of this work is to give a constructive proof for a 
particular case of the no-name lemma. Let $G$ be a
finite group, $K$ be a field, $L$ be a permutation $G$-lattice 
and $K[L]$ be the group algebra of $L$ over
$K$. The no-name lemma asserts that the invariant field of the
quotient field of $K[L]$, $K(L)^G$ is a purely transcendental
extension of $K^G$. In other words, there exist $y_1, \ldots , y_n$
which are algebraically independent over $K^G$ such that $K(L)^G \cong
K^G(y_1, \ldots , y_n)$.  We define elements $\lbrace y_1, \ldots, y_n
\rbrace \subset K[L]^G$ with the desired properties, in the case when
$G$ is the Galois group of a finite extension $\mathrm{Gal}(K/F)$, and
%is given as a group of (signed) permutation matrices.
$L$ is a sign permutation $G$-lattice.
\end{abstract}

\section{Introduction}

An algebraic $F$-torus $T$ is an algebraic group defined over a field
$F$ which {\em splits} over an algebraic closure $\bar F$ of $F$, that
is, which is isomorphic to a torus (a finite product of copies of the
multiplicative group $\mathbb{G}_m$) over $\bar{F}$. In general,
$\bar{F}$ is not the smallest field over which $T$ splits: it is known
that an algebraic $F$-torus $T$ splits over a finite Galois extension
of $F$. There is a unique minimal such extension, say $K$; if $G =
\mathrm{Gal}(K/F)$, then $G$ is called the splitting group of $T$. For
more details, 
%see~\cite[Chapter 2]{Voskresenskii}.
see~\cite[p. 27]{Voskresenskii}.

For a finite group $G$, a $G$-lattice $L$ is a free $\Z$-module $A$ of
finite rank, together with a group homomorphism $G \longrightarrow
{\rm Aut}(A)$ (the group of automorphisms of $A$). Given a module
basis of $A$, any group homomorphism $G \longrightarrow {\rm
  GL}(n,\Z)$, with $n={\rm rank}(A)$, gives such an action.  If $K$ is
a field, the group algebra $K[L]$ of $L$ over $K$ is isomorphic to the
$K$-algebra of Laurent polynomials $K[x_1^{\pm 1},\dots,x_n^{\pm 1}]$,
for some indeterminates $x_1,\dots,x_n$.  If $K$ is equipped with an
action of $G$ (that is, a $G$-field), we can extend the action of $G$
on lattice $L$ to an action on $K[L]$; the ring $K[L]^G$ of {\em
  multiplicative invariants} consists of those elements in $K[L]$
invariant under the action of $G$. The fraction field $K(L)$ of $K[L]$
is isomorphic to $K(x_1,\dots,x_n)$, and the subfield of invariants
under the action of $G$ is written $K(L)^G$.

It is known that there is a duality between the category of algebraic
tori with splitting group $G$ and $G$-lattices. For a given algebraic
torus $T$ with splitting group $G$, its character module
$\mathrm{Hom}(T,\mathbb{G}_m)$ is a $G$-lattice. Conversely, if $L$ is
a $G$-lattice, with $G={\rm Gal}(K/F)$ for some finite Galois
extension $K/F$, then $T=\mathrm{Spec}(K[L]^G)$ is an algebraic
$F$-torus with splitting group $G$, coordinate ring $K[L]^G$ and
function field $K(L)^G$.

A $G$-lattice $L$ is called {\em permutation} (resp.\ {\em sign
  permutation}) if it has a $\Z$-basis which is permuted (resp.\ up to
sign changes) by $G$. In particular, an algebraic torus whose
corresponding $G$-lattice is a permutation lattice is called a {\em
  quasi-split torus}. Quasi-split algebraic tori 
  %Algebraic tori
  are characterized as being
representable as a direct product of groups of the form
$R_{E/F}(\mathbb{G}_m)$, where $R_{E/F}$ is the Weil restriction with 
respect to a finite separable field extension $E/F$.
Note that $R_{E/F}(\mathbb{G}_m)$ is a $F$-group scheme with $F$-points $E^{\times}$ and 
character lattice $\Z[G/H]$, where $K/F$ is a Galois closure of $E/F$, $G={\rm Gal}(K/F)$ and $H$ is the 
subgroup of $G$ which fixes the subfield $E$.
%(so
%that $R_{K/F}(\mathbb{G}_m)$ is an $F$-group scheme).
The Weil restriction here is not necessarily with respect to a finite Galois extension $K/F$
as this would only produce character lattices which are direct sums of the group ring.

 %%to which we refer by the same name, 

A specific case of the no-name lemma asserts that if $L$ is a
permutation $G$-lattice and $K$ is a $G$-field, then $K(L)^G$ is
rational over $K^G$~\cite[Chapter~9.4]{Lorenz}; in particular, with
$G={\rm Gal}(K/F)$, $K(L)^G$ is $F$-rational. The term "no-name lemma"
was first used by Dolgachev in \cite{Dolgachev}, expressing the fact
that many researchers discovered the result independently. It is
actually more general than the stated version here; see
\cite[p.\ 6]{Dolgachev}, \cite[Section 3.2]{Sansuc}, \cite[Proposition
  1.3]{Lenstra}, \cite[Remark 2.4]{Domokos} and \cite[Proposition
  1.1]{EndoMiyata}.  In this paper, we give a constructive proof of
the particular case described above (and of a slight generalization
thereof, using signed permutation matrices), by exhibiting a basis for
such a field of invariants.  In concrete terms, we start from a
subgroup of $\mathrm{GL}(n,\Z)$ and describe the field of functions of
an associated torus.

\begin{definition}\label{Assumption}
  Let $G$ be a finite subgroup of $\mathrm{GL}(n,\Z)$.  The
  $G$-lattice $L_G$ corresponding to $G$ is the rank $n$ lattice
  $\Z^n=\{[a_1,\dots,a_n]^T: a_i\in \Z\}$ on which $G$ acts naturally by left-multiplication.
  Note that $\Z^n$ has a standard basis $\{\be_i: i=1,\dots,n\}$, 
  %generated by the standard basis $\langle \be_i : i = 1,
  %\ldots, n \rangle_\Z$, 
  where $\be_i$ is the column vector
  $[\delta_{i,j}, i=1,\dots,n]^T$.
  %together with the action of $G$ given
  %by left-multiplication on the $\be_i$'s.
\end{definition}

Suppose further that we are given an isomorphism $\iota: G \to
\mathrm{Gal}(K/F)$, for some finite Galois extension $K/F$ (in what
follows, we simply say that $K/F$ has Galois group $G$). Then, through
the identification $K[x_1^{\pm 1}, \ldots , x_n^{\pm 1}]\simeq K[L_G]$,
$K[x_1^{\pm 1}, \ldots , x_n^{\pm 1}]$ is equipped with the $G$-action
defined as follows:
\begin{itemize}
\item for $g$ in $G$ and $\alpha$ in $K$, we write
  $g(\alpha)=(\iota(g))(\alpha)$ (so $G$ acts as the Galois group on
  $K$);
\item for $g$ in $G$ and $j=1,\dots,n$, $g(x_i) = \prod_{j=1}^{n}
  x_j^{g_{j,i}}$, where $g_{j,i}$ is the $(j,i)$th -entry of $g$
 (so that we also have $g(\be_i) =
  \sum_{j=1}^{n} g_{j,i}\be_j$).
\end{itemize}
If we let $T_G$ be the algebraic torus corresponding to $L_G$, then
$T_G$ is an algebraic $F$-torus which splits over $K$, with character
lattice $L_G$ and function field $K(x_1,\dots,x_n)^G$.

Conjugate subgroups of $\mathrm{GL}(n,\Z)$ correspond to isomorphic
lattices, and isomorphic algebraic tori; in particular, for $G$ a
finite subgroup of $\mathrm{GL}(n,\mathbb{Z})$, $L_G$ is a (signed)
permutation lattice if and only if $G$ is conjugate to a group of
(signed) permutation matrices.  Computationally, we do not have an
efficient algorithm at hand to decide whether a given lattice is
(signed) permutation. Hence, in our main results, we will assume that
$G$ is a subgroup of the group $\mathbb{S}_n$ of permutation matrices
of size $n$, or more generally of the group $\mathbb{B}_n$ of signed
permutation matrices of size $n$.  In such a case, for $i$ in
$\{1,\dots,n\}$ and $g$ in $G$, $g(\be_i)=\pm \be_j$ for some index
$j$ in $\{1,\dots,n\}$ (all signs being $+1$ if $G$ is a subgroup of
$\mathbb{S}_n$), and the action of $g \in G$ on $x_i$ is given
by $$g(x_i) = \begin{cases} x_j & \text{if~} g(\be_i) = \be_j
  \\ x^{-1}_j & \text{if~} g(\be_i) = -\be_j. \end{cases}$$

For such groups $G$, the $F$-rationality of the torus $T_{G}$ means
that for $K(x_1,\dots,x_n)$, endowed with the $G$-action we just
described, there exist algebraically independent $y_1,\dots,y_n$ in 
$K(x_1,\dots,x_n)$ such that $K(x_1,\dots,x_n)^G=F(y_1,\dots,y_n)$. 
However, the proofs of the no-name lemma we are aware of
are nonconstructive.  The goal of this paper is to exhibit such 
an algebraically independent set $\{y_i:i=1,\dots,n\}$; we
state two such results. Note that the first result is a special case of
 the general version of the no-name lemma. In particular, this can be 
 applied to get an explicit transcendence basis of the function field 
 of a quasi-split algebraic torus $T_G$, provided $G$ is given as a group 
 of permutation matrices.

In both our theorems, we rely on the notion of a {\em normal element} of
a finite Galois extension $K/F$ with Galois group $G$; we recall that
$\alpha \in K$ is normal if $\alpha$ and all its Galois conjugates
form an $F$-basis of $K$.  Any finite Galois extension admits a normal
element~\cite[Theorem 6.13.1]{Lang}; there exist algorithms to
construct one, in characteristic zero~\cite{Girstmair} or in positive
characteristic~\cite{Giesbrecht,Poli}.

\begin{theorem}\label{nonamenonsign}
  Let $G$ be a subgroup of $\mathbb{S}_n$, let $K/F$ be a finite
  Galois extension with Galois group $G$, and let $\alpha \in K$ be a
  normal element for $K/F$. Then $ K(x_1,\ldots ,
  x_n)^G=F(y_1,\dots,y_n)$, with
  $$y_i =\sum_{g \in G} g(\alpha x_i), \quad i = 1,\dots,n.$$
\end{theorem}

Our second statement is similar, but deals with the more general case
of signed permutations (if the group $G$ below happens to be a subgroup of $\mathbb{S}_n$,  
%non-signed permutation, 
the construction is not the same as in the previous theorem).

\begin{theorem}\label{nonamesign}
  Let $G$ be a subgroup of $\mathbb{B}_n$, let $K/F$ be a finite
  Galois extension with Galois group $G$, and let $\alpha \in K$ be a
  normal element for $K/F$. Then
  $K(x_1,\ldots, x_n)^G = F(y_1, \ldots, y_{n}),$
  with $$y_i = \sum_{g \in G} g\left (\frac{\alpha}{1+x_i}\right), \quad i = 1,\dots,n.$$
\end{theorem}

There are many algorithms for finding the invariant rings for
polynomial invariants~\cite{Sturmfels,Kemper2}. For multiplicative
invariants, the algorithmic landscape is not developed to the same
extent. In \cite{Renault} the author introduced an algorithm to 
compute a generating set for the ring of multiplicative invariants, 
when the acting group is a subgroup of a reflection group. A more 
general algorithm for computing the ring of multiplicative invariants
 is given by Kemper in \cite{Kemper}. It is worth mentioning that although 
both above algorithms may be applied to our problem, they 
will not necessarily produce a transcendence basis of the invariant field.

\paragraph{Acknowledgements.} A large part of this work was done while the first 
author was a PhD student at the University of Western Ontario. The second and 
third authors are supported by an NSERC Discovery Grant. We thank Gregor Kemper and 
Lex Renner for their feedback on an earlier version of this work.

%%%%%%%%%%%%%%%%%%%%%%%%%%%%%%%%%%%%%%%%%%%%%%%%%%%%%%%%%%%%
%%%%%%%%%%%%%%%%%%%%%%%%%%%%%%%%%%%%%%%%%%%%%%%%%%%%%%%%%%%%
%%%%%%%%%%%%%%%%%%%%%%%%%%%%%%%%%%%%%%%%%%%%%%%%%%%%%%%%%%%%

\section{Proofs and examples}

In what follows, we use the following notation: we still write
$\mathbb{S}_n$, resp.\ $\mathbb{B}_n$, for the groups of permutation,
resp.\ signed permutation matrices of size $n$. Note that $\mathbb{D}_n$, the group of  diagonal matrices 
in $\mathrm{GL}_n(\Z)$ (diagonal matrices with entries in $\{\pm 1\}$) is a normal subgroup of $\mathbb{B}_n$
such that $\mathbb{B}_n=\mathbb{D}_n\rtimes \mathbb{S}_n$ is a semidirect product.
We let $\mathfrak{S}_n$ be the symmetric group of size $n$.  Since $\mathbb{S}_n$ is naturally isomorphic to 
$\mathfrak{S}_n$, we see that there is a natural group homomorphism  
$\rho: \mathbb{B}_n \to \mathfrak{S}_n$ 
obtained by mapping a signed permutation matrix to the permutation
$\rho_g$ such that $\rho_g(i)=j$, where $j$ is the index of the
unique non-zero entry in the $i$th column of $g$.  Note that the kernel of the group homomorphism is $\mathbb{D}_n$. %(\we 
%$\mathfrak{S}_n$ be the symmetric group of size $n$ and we denote by
%$\rho: \mathbb{B}_n \to \mathfrak{S}_n$ the group homomorphism
%obtained by mapping a signed permutation matrix to the permutation
%$\rho_g$ such that $\rho_g(i)=j$, where $j$ is the index of the
%unique non-zero entry in the $i$th column of $g$.  
Hence, in terms of
the action defined in the previous section, for $g$ in $\mathbb{B}_n$
and for all $i,j$ in $\{1,\dots,n\}$, we have $g(x_i) =
x_{\rho_g(i)}^{\pm 1}$.  For $i,j$ as above, we also denote by
$G^{\pm}_{i,j}$ the set of all $g$ in $G$ such that $g(x_i)=x^{\pm 1}_j$, that is,
such that $\rho_g(i)=j$.
%% \textbf{Changed notation here to not duplicate permutation case notation}.

We start with a lemma that generalizes known facts about Moore
matrices over finite fields (see Example~\ref{ex:Moore} below). Let
$G$ be a subgroup of $\mathbb{B}_n$ and let $K/F$ be a finite Galois
extension with Galois group isomorphic to $G$, as in the previous
section. Through this isomorphism, $G$ acts on (column) vectors
entrywise: for $g \in G$ and a column vector in $K^m$, written as $C =
\left[\mu_{1} \,\,\, \mu_{2} \,\,\, \cdots \,\,\, \mu_{m}\right]^T$,
we define $g(C) = \left[g(\mu_{1}) \,\,\, g(\mu_{2}) \,\,\, \cdots
  \,\,\, g(\mu_{m})\right]^T$.

Let then $M$ be in $M_{n,n}(K)$. For such a matrix, and for
$i=1,\dots,n$, its $i$th column is written $M_i=\left[\mu_{1,i} \,\,\,
  \mu_{2,i} \,\,\, \cdots \,\,\, \mu_{m,j}\right]^T$.  We say that $G$
{\em permutes the columns of $M$ up to sign} if for $g$ in $G$ and $i$
in $\{1,\dots,n\}$, we have $g(M_i)= \pm M_{\rho_g(i)}$.
 
\begin{lemma}\label{signdet}
   Let $K/F$ be a finite Galois extension with Galois group $G$. Let
   $M$ be in $M_{n,n}(K)$ and assume that $G$ permutes the columns of
   $M$ up to sign. Assume also that the entries of the first column
 of $M$ are $F$-linearly independent. Then $M$ is invertible.
\end{lemma}
\begin{proof}
  Assume by contradiction that there is a non-zero vector in the left
  nullspace of $M$; take $\bx \in K^n$ to be a vector with the minimum number of
  non-zero entries among the non-zero left nullspace elements
  (that is, such that $\bx^T M=0$).
  Let $k \in \{1,\dots,n\}$ be such that $x_k \neq 0$ and let
  $\by=\frac 1{x_k} \bx \in K^n$, so that $y_k = 1$, and $\by$ is
  still in the left nullspace of $M$. 
  
  For $i$ in $\{1,\dots,n\}$, we have the equality $\by^T M_i = 0 $,
  where $M_i$ is the $i$th column of $M$. For $g$ in $G$, we deduce
  $g(\by^T M_i) = g(\by)^Tg(M_i) = \pm g(\by)^T M_{\rho_g(i)} = 0$. Since
  this is true for all $i$, we obtain that $g(\by)$ is in the
  left nullspace of $M$ as well. This further implies that
  $\by'=g(\by)-\by$ is in the left nullspace of $M$. However, since
  $y_k=1$, $g(y_k)=1$, so that $y'_k=0$. By construction of $\bx$, this
  implies that $\by'=0$, so that $g(\by)=\by$. 

  Since this is true for all $g$, we deduce that $\by$ is in $F^n$.
  Then, the relation $\by^T M_1=0$ implies that $\by=0$, a
  contradiction.
 \end{proof} 

\begin{remark}
The hypotheses in the above lemma imply that the induced unsigned
action of the group $G$ on the columns of $M$ is a transitive action.
Let $H=\mathrm{Stab}_G(M_1)$. Then $h\in H$ fixes $M_1$ and so $M_1\in
(K^H)^n$.  But then the $n$ entries of $M_1$ are contained in $K^H$
and are also $F$-linearly independent.  The size of an $F$-linearly
independent set in $K^H$ can be at most $[K^H:F]$.  Then $n\le
[K^H:F]=[K^H:K^G]=[G:H]\le n$ implies that $[G:H]=n$.  But then the
orbit of $M_1$ under the action of $G$ has size
$[G:H]=[G:\mathrm{Stab}_G(M_1)]=n$.
\end{remark}  

\begin{example}\label{ex:Moore}
  Let $F=\F_q$, for some prime power $q$, and let $K=\F_{q^n}$.  The
  Galois group of $K/F$ is cyclic of size $n$, generated by the Frobenius
  map $x \mapsto x^q$. Let then $(\alpha_1,\dots,\alpha_n)$ be in $K$,
  and consider the {\em Moore matrix} $M=[m_{i,j}]_{1 \le i,j \le n}$, with
  $m_{i,j} = \alpha_i^{q^{j-1}}$. The Frobenius map permutes the
  columns of $M$, and we recover the fact that if
  $(\alpha_1,\dots,\alpha_n)$ are $F$-linearly independent, $M$ is
  invertible~\cite[Corollary~1.3.4]{Goss}.
\end{example}

We can now prove our first result.

%% \begin{theorem}\label{theo:1}
%%   Let $G$ be a subgroup of $\mathbb{S}_n$, let $K/F$ be a finite
%%   Galois extension with Galois group $G$, and let $\alpha \in K$ be a
%%   normal element for $K/F$. Then $ K(x_1,\ldots ,
%%   x_n)^G=F(y_1,\dots,y_n)$, with
%%   $$y_i =\sum_{g \in G} g(\alpha x_i), \quad i =
%%   1,\dots,n.$$
%% \end{theorem}
\begin{proof}[Proof of Theorem~\ref{nonamenonsign}]
  We first prove the result under the additional assumption that $G$
  acts transitively on $L_G$.  The elements $(y_1,\dots,y_n)$, with
  $y_i = \sum_{g \in G} g(\alpha x_i)$ as defined in the theorem, are
  invariant under the action of $G$. We will show below that
  $K(x_1,\dots,x_n)=K(y_1,\dots,y_n)$; this will prove that
  $K(x_1,\dots,x_n)^G=F(y_1,\dots,y_n)$, since
  $K(x_1,\dots,x_n)=K(y_1,\dots,y_n)$ implies that
  $K(y_1,\dots,y_n)^G=K^G(y_1,\dots,y_n)=F(y_1,\dots,y_n)$.

  For $i,j$ in $\{1,\dots,n\}$, let $G_{i,j}= \lbrace g \in G: g(x_i)
  = x_j \rbrace$, so that we can rewrite $y_i$ as $$y_i = \sum^n_{j
    =1}\sum_{g \in G_{i,j}}g(\alpha)x_j, \,\,\, i = 1, \ldots, n.$$
  Note that $G_{i,i}=\mathrm{Stab}_G(x_i)$ and so is a subgroup of $G$. Since the action of $G$ on $L_G$ is 
  transitive, $G_{i,j}$ is
  non-empty for every $1 \leq i,j \leq n$. Take such indices $i,j$,
  and fix some $g_{i,j}$ in $G_{i,j}$. If $g \in G_{i,j}$, then
  $g^{-1}_{i,j}g(x_i) = x_i$ shows that $g$ is in
  $g_{i,j}G_{i,i}$. Since we also have
  $g_{i,j}G_{i,i} \subseteq G_{i,j}$, we see that
  $G_{i,j} = g_{i,j}G_{i,i}$.

 % Since the action of $G$ on $L_G$ is transitive, $G_{i,j}$ is
  %non-empty for every $1 \leq i,j \leq n$.  Take such indices $i,j$,
  %and fix some $g_{i,j}$ in $G_{i,j}$. If $g \in G_{i,j}$, then
  %$g^{-1}_{i,j}g(x_i) = x_i$ shows that $g$ is in
  %$g_{i,j}\mathrm{Stab}_G(x_i)$. Since we also have
  %$g_{i,j}\mathrm{Stab}_G(x_i) \subseteq G_{i,j}$, we see that
  %$G_{i,j} = g_{i,j}\mathrm{Stab}_G(x_i)$.

  We now show that the matrix $M$ with $i$th row the coordinate vector
  of $y_i$ with respect to the $K$-basis $\lbrace x_1, \ldots x_n
  \rbrace$ is invertible. The matrix $M$ has entries $m_{i,j} =
  \sum_{g \in G_{i,j}}g(\alpha)$, $i,j = 1, \ldots, n$. We will apply
  Lemma~\ref{signdet} to show that $M$ is invertible, which is sufficient to 
  prove the theorem.

  We check the hypothesis of the lemma. First, let $\rho: G \to
  \mathfrak{S}_n$, $\rho(g) = \rho_g$ be the group homomorphism that
  corresponds to the action of $G$ on $(x_1, \ldots , x_n)$, so that
  $\rho_g(i) = j$ if and only if $g(x_i) = x_j$ for all $1 \leq
  i,j\leq n$. We will show that the columns of $M$ are permuted by the
  action of $G$. Let thus $h$ be in $G$. Note that for $g$ in
  $G_{i,j}$, $hg$ is in $G_{i,\rho_h(j)}$.
  Then since $G_{i,j}=g_{i,j}G_{i,i}$ is a left coset of $G_{i,i}$ where $g_{i,j}$
is an arbitrary element of $G_{i,j}$, we see that $hG_{i,j}=hg_{i,j}G_{i,i}=G_{i,\rho_h(j)}$ 
since $hg_{ij}\in G_{i,\rho_h(j)}$.  We then get
  %since $G_{i,\rho_h(j)}$ and
  %$G_{i,j}$ have the same cardinality, equal to
  %$|\mathrm{Stab}_G(x_i)|$, we get
  $$h(m_{i,j}) = \sum_{g \in G_{i,j}}hg(\alpha) = \sum_{\sigma \in
    G_{i, \rho_{h}(j)}}\sigma (\alpha) = m_{i,\rho_{h}}(j).$$ This
  shows that $h(M_j) = M_{\rho_h(j)}$ for all $j = 1, \ldots, n$, so
  that $G$ permutes the columns of $M$. 
  
  Finally, the first column $M_1$ has entries $\sum_{g \in
      G_{i,1}}g(\alpha)$, $i = 1, \ldots,n$.  Since $\alpha$ is a
  normal element of the Galois extension $K/F$ with Galois group $G$,
  the set $\lbrace g(\alpha): g \in G \rbrace$ is $F$-linearly
  independent. Since $G = \sqcup^n_{i =1}G_{i,1}$ is a disjoint union,
  and all $G_{i,1}$ are non-empty,
  the set
  $$\left\lbrace \sum_{g \in G_{i,1}}g(\alpha),\ i = 1, \ldots, n
  \right\rbrace$$ is $F$-linearly independent as well.  So
  Lemma~\ref{signdet} applies, and we conclude that $M$ is invertible,
  as claimed.

  \medskip 
  
  We can now give the proof of our claim in the general case.  Let
  $\be_1, \ldots , \be_n$ be the standard basis of $L_G$, and let
  $\{\be_{j_k} \mid k = 1, \ldots, r\}$ and correspondingly $\{x_{j_k}
  \mid k = 1, \ldots, r\}$ be a complete set of $G$-orbit
  representatives among the basis vectors, and the indeterminates
  $x_1, \ldots, x_n$ respectively. Then $L_k = \oplus_{ \be_i \in G
    \be_{j_{k}} } \Z \be_i$ is a transitive permutation $G$-lattice
  for each $k = 1, \ldots, r$, and $K(L_k) = K(x_i \mid x_i \in
  Gx_{j_k})$.
  
  The lattice $L_G = \oplus^r_{k =1} L_k$ is a direct sum of transitive
  permutation $G$-lattices, 
  %and similarly 
  so that $K(x_1,\dots,x_n)$
  %Note superscript G was removed.
  is the
  compositum of the fields $K(L_k)$, $k=1,\dots,r$. Thus, using the 
  result established in the transitive case, we obtain
  $K(x_1,\dots,x_n)^G=F(y_i \mid x_i \in Gx_{j_k}, k=1,\dots,r)$, where
  for all $k$ and for $x_i \in Gx_{j_k}$, we have $y_i=\sum_{g \in G}
  g(\alpha x_i)$.
\end{proof}

\begin{example}
Let $K$ be the splitting field of $x^4-2$ over $F=\Q$. Then
$\mathrm{Gal}(K/\Q) \cong D_8$, $K = \Q(\theta,i)$, with $\theta =
\sqrt[4]{2}$, and $\lbrace 1, \theta, \theta^2, \theta^3, i, i\theta,
i\theta^2, i\theta^3\rbrace$, is a $\Q$-basis for $K$. Let
$n=4$, let
$G\leq \mathrm{GL}(4,\Z)$ be generated by
$$
r = \begin{bmatrix}
0&0&0&1\\
1&0&0&0\\
0&1&0&0\\
0&0&1&0
\end{bmatrix} \,\,\,\, \text{and}\,\,\,\,
s = \begin{bmatrix}
1&0&0&0\\
0&0&0&1\\
0&0&1&0\\
0&1&0&0
\end{bmatrix}.
$$ and let $(x_1,\dots,x_4)$ be new indeterminates, on which $G$ acts as in
Definition~\ref{Assumption}; this action is transitive. One can verify
that $G$ is isomorphic to $\mathrm{Gal}(K/\Q)$; through this isomorphism,
the action of $r$ and $s$ on the generators of $K$ is given by
$$
\begin{matrix}
r(i) = i &r(\theta) = i \theta \\
s(i) = -i & s(\theta) = \theta.
\end{matrix}
$$
Now, define $$\alpha = 1+ \theta + \theta^2 + \theta^3 +i + i\theta + i\theta^2 + i\theta^3=
(\sum_{i=0}^3\theta^i)(1+i);$$
this is a normal element in $K/\Q$. 
Note that $G_{1,1}=\mathrm{Stab}_G(x_1)=\langle s\rangle$. Since $r^i(x_1)=x_{i+1}, i=0,1,2,3$,
we see that $G_{i,i}=\mathrm{Stab}_G(x_i)=r^i \mathrm{Stab}_G(x_1)r^{-i}$ and so $G_{1,1}=G_{3,3}=\langle s\rangle$
and $G_{2,2}=G_{4,4}=\langle r^2s\rangle$. Also note that $r^{j-i}\in G_{i,j}$, where we may consider all exponents of 
$r$ modulo 4.
This shows that the elements $(y_1,\dots,y_4)$ of Theorem~\ref{nonamenonsign}, expressed 
on the basis $(x_1,\dots,x_4)$, are given by the coordinate matrix
$$
M= \begin{bmatrix}
(1+s)(\alpha) & r(1+s)(\alpha) & r^2(1+s)(\alpha) & r^3(1+s)(\alpha)\\
r^3(1+r^2s)(\alpha) & (1+r^2s)(\alpha) & r(1+r^2s)(\alpha) & r^2(1+r^2s)(\alpha)\\
r^2(1+s)(\alpha) & r^3(1+s)(\alpha) & (1+s)(\alpha) & r(1+s)(\alpha)\\
r(1+r^2s)(\alpha) & r^2(1+r^2s)(\alpha) & r^3(1+r^2s)(\alpha) & (1+r^2s)(\alpha)
\end{bmatrix}.
$$
Now since $M_j=r^{j-1}M_1$ for all $j=1,2,3,4$, it is clear that the action of 
$r$ permutes the columns.  In particular, $rM_j=M_{j+1}$ (modulo 4).  One can check that $sM_1=M_1$ and since $|\textrm{Stab}_G(M_1)|=8/4=2$, this shows that $\textrm{Stab}_G(M_1)=\langle s\rangle$.
Then $sr^{j-1}M_1=r^{1-j}sM_1$ shows that $sM_j=M_{5-j},j=1,\dots,4$ so $s$ also permutes the columns of $M$.
%$$
%M= \begin{bmatrix}
%(1+s)(\alpha) & (r^3+sr)(\alpha) & (r^2+sr^2)(\alpha) & (r+sr^3)(\alpha)\\
%(r+sr)(\alpha) & (1+sr^2)(\alpha) & (r^3+sr^3)(\alpha) & (r^2+s)(\alpha)\\
%(r^2+sr^2)(\alpha) & (r+sr^3)(\alpha) & (1+s)(\alpha) & (r^3+sr)(\alpha)\\
%(r^3+sr^3)(\alpha) & (r^2+s)(\alpha) & (r+sr)(\alpha) & (1+sr^2)(\alpha)
%\end{bmatrix}.
%$$
%The action of $r$ and $s$ on the columns is given by
%% $$r(M_1)=M_4,\quad r(M_2)=M_1,\quad r(M_3)=M_2, \quad r(M_4)=M_3$$
%% and
%% $$s(M_1)=M_1,\quad s(M_2)=M_4,\quad s(M_3)=M_3, \quad s(M_4)=M_2.$$
%\begin{table}[H]
%\centering
%\begin{tabular}{l|llllllll} 
% & $r$ & $s$ \\
% \hline
% $M_1$  & $M_4$ & $M_1$ \\
%$M_2$ & $M_1$ & $M_4$ \\
%$M_3$ & $M_2$ & $M_3$ \\
%$M_4$ & $M_3$ & $M_2$ 
%\end{tabular}
%\end{table}
\end{example} 

\begin{remark}
With the assumptions of the previous theorem, we can actually compute
the coordinate ring of the torus; we obtain $$K[L_G]^G \cong K[x^{\pm
    1}_1, x^{\pm 1}_2, \ldots , x^{\pm 1}_n]^G = F[y_1, \ldots ,
  y_n]_{x_1\cdots x_n},$$ for $y_1,\dots,y_n$ as in the theorem.
Indeed, we have $K[x^{\pm 1}_1, x^{\pm 1}_2, \ldots , x^{\pm 1}_n] =
K[x_1, \ldots , x_n]_{x_1\cdots x_n}.$ We are interested in
$K[L_G]^G$, that is, $\left( K[x^{\pm 1}_1, x^{\pm 1}_2, \ldots ,
  x^{\pm 1}_n] \right)^G = \left(K[x_1, \ldots , x_n]_{x_1\cdots x_n}
\right) ^G.$ The proof of Theorem~\ref{nonamenonsign} shows that
$K[x_1, \ldots , x_n] = K[y_1, \ldots , y_n]$.  On the other hand
since $G$ permutes the $x_i$'s, $x_1\cdots x_n$ is invariant under the
action of $G$, and we can conclude $$\left( K[x_1, \ldots ,
  x_n]_{x_1\cdots x_n}\right)^G = \left( K[y_1, \ldots ,
  y_n]_{x_1\cdots x_n} \right)^G = K^G [y_1, \ldots , y_n]_{x_1\cdots
  x_n} = F[y_1, \ldots , y_n]_{x_1\cdots x_n}.$$ One could further
rewrite $x_1\cdots x_n$ as a degree $n$ homogeneous polynomial in
$y_1,\dots,y_n$ (but the expression obtained this way is not
particularly handy).
\end{remark}

We conclude with the proof of our second main result.  The proof
follows that of Theorem~\ref{nonamenonsign}, the only difference being in
the description of the coordinate matrix $M$.  As in
Theorem~\ref{nonamenonsign}, we first prove the result under the extra
assumption that $G$ acts transitively up to sign on $L_G$.

\begin{proof}[Proof of Theorem~\ref{nonamesign}]
  Assume first that the action of $G$ is transitive (up to sign).
  For $i$ in $\{1,\dots,n\}$, define $z_i = (1+x_i)^{-1} $. Then, for $g
  \in G$, $$g(z_i) = \begin{cases} z_j & \text{if} \,\,\, g(x_i) = x_j
    \\ 1-z_j & \text{if} \,\,\, g(x_i) = x_j^{-1},
  \end{cases}$$ 
  and $K(x_1, \ldots , x_n) = K(z_1, \ldots, z_n).$ The elements $y_i$
  can be rewritten as $y_i = \sum_{g \in G} g ({\alpha}z_i)$, for $i$
  in $\{1, \ldots, n\}$; as before, in order to prove that
  $K(z_1,\ldots, z_n)^G = F(y_1, \ldots, y_{n})$, it is enough to
  prove that $K(y_1,\dots,y_n)=K(z_1,\dots,z_n)$. This will be done by
  writing $(1,y_1,\dots,y_n)$ as $K$-linear combinations of
  $(1,z_1,\dots,z_n)$, and proving that the coordinate matrix is
  invertible.

  For $i,j$ in $\lbrace 1, \ldots , n \rbrace$, let $G^{\pm}_{i,j}$
  be defined as in the preamble of this section, that is 
  $G^{\pm}_{i,j}=\lbrace g \in G : g(z_i) = z_j \,\, \text{or} \,\, g(z_i) = 1-z_j
  \rbrace $. By the transitivity assumption, $G^{\pm}_{i,j}$ is
  non-empty for every $1 \leq i,j \leq n$. Let us further define $G^{+}_{i,j}
  = \lbrace g \in G : g(z_i) = z_j \rbrace$ and $G^{-}_{i,j} =
  \lbrace g \in G : g(z_i) = 1- z_j \rbrace$, so that $G_{i,j}=
  G^{+}_{i,j}\sqcup G^{-}_{i,j}$.  

Note that $G^+_{i,i}=\textrm{Stab}_G(z_i)$.  Both $G^{\pm}_{i,i}$ and $G^+_{i,i}$ are subgroups of $G$.  One can show that the left cosets of $G^+_{i,i}$ in $G$ are the non-empty sets in the collection
$$\{G^+_{i,j}:1\le j\le n\} \cup \{G^-_{i,j}: 1\le j\le n\}$$
By the transitivity assumption, the sets $G^{\pm}_{i,j}\ne \emptyset$
so one can guarantee that either $G^+_{i,j}\ne \emptyset$ or $G^-_{i,j}\ne \emptyset$ (although both are possible).  If $G^+_{i,j}\ne \emptyset$, and $g_{ij}^+$ is an arbitrary element of $G^+_{i,j}$, then 
$G^+_{i,j}=g_{ij}^+G^+_{i,j}$ and if $G^-_{i,j}\ne \emptyset$ and $g_{ij}^-$ is an arbitrary element of $G^-_{i,j}$, then 
$G^-_{i,j}=g_{ij}^-G^+_{i,j}$.
  
  Let $M^*$ be the coordinate matrix of $(1, y_1, \ldots,
  y_n)$ with respect to the $K$-basis $(1, z_1, \ldots, z_n)$; 
  we  have to show that $\det (M^*) \neq 0$.
  By definition, for $i$ in $\{1,\dots,n\}$, we have
  \begin{align*}
y_i = \sum_{g \in G} g ({\alpha}z_i)&= \sum_{j=1}^n \Big(\sum_{g\in G^{+}_{i,j}}g(\alpha)z_j +\sum_{g\in G^{-}_{i,j}}g(\alpha)(1-z_j)\Big)\\
&=\sum_{j=1}^n\sum_{g\in G^{-}_{i,j}}g(\alpha)+ \sum_{j=1}^n\Big(\sum_{g\in G^{+}_{i,j}}g(\alpha) -\sum_{g\in G^{-}_{i,j}}g(\alpha)\Big)z_j.
  \end{align*}
For $i,j \in \lbrace1, \ldots , n \rbrace$, define $m_{i,j} =
\sum_{g\in G^{+}_{i,j}}g(\alpha) -\sum_{g\in G^{-}_{i,j}}g(\alpha)$
and $c_i = \sum_{j=1}^n\sum_{g\in G^{-}_{i,j}}g(\alpha)$. The matrix
$M^*$ is then
$$M^* = \begin{bmatrix}
1 & 0 & \cdots & 0\\
c_1 & m_{1,1} & \cdots & m_{1,n}\\
\vdots & \vdots &  & \vdots\\
c_{n} & m_{n,1} & \cdots	& m_{n,n}
\end{bmatrix}.
$$
Let us write
$$M = \begin{bmatrix}
 m_{1,1} & \cdots & m_{1,n}\\
 \vdots &  & \vdots\\
 m_{n,1} & \cdots	& m_{n,n}
\end{bmatrix}.$$
Since $\det(M^*) = \det (M)$, it is enough to show that the
determinant of $M$ is non-zero; this will be done using
Lemma~\ref{signdet}. We now check that the hypotheses of the lemma are
satisfied.

As before, let $\rho: G \to \mathfrak{S}_n$, $\rho(g) =
\rho_g$ be the group homomorphism that corresponds to the action of
$G$ on $\lbrace z_1, \ldots , z_n \rbrace$, so that $\rho_g(i) = j$ if
and only if $g$ is in $G^{\pm}_{i,j}$. We will show that the columns
$M_1,\dots,M_n$ of $M$ are permuted up to sign by the action of $G$.

Let $h$ be in $G$
and  $i,j$ be in $\lbrace 1, \ldots, n \rbrace$. We 
can then write
$$h(m_{i,j}) = h\Big( \sum_{g\in G^{+}_{i,j}}g(\alpha) -\sum_{g\in
  G^{-}_{i,j}}g(\alpha)\Big) = \sum_{g\in G^{+}_{i,j}}hg(\alpha)
-\sum_{g\in G^{-}_{i,j}}hg(\alpha).$$ 
As in the proof of Theorem~\ref{nonamenonsign}, we have
$hG^{\pm}_{i,j} = G^{\pm}_{i,\rho_h(j)}$, but more precisely, we can write
\begin{align}
\left \{
\begin{array}{ll}
  G^{+}_{i,\rho_{h}(j)}&= hG^{+}_{i,j}\\
G^{-}_{i,\rho_h(j)}&= hG^{-}_{i,j}
\end{array}\right .
\text{~if~} h \in G^+_{j,\rho_h(j)}
\quad\text{and}\quad
\left \{
\begin{array}{cl}
  G^{+}_{i,\rho_{h}(j)}&= hG^{-}_{i,j}\\
G^{-}_{i,\rho_h(j)}&= hG^{+}_{i,j}
\end{array}\right .
\text{~if~} h \in G^-_{j,\rho_h(j)}.
\end{align}
In the first case, we deduce
$$m_{i,\rho_h(j)} =  \sum_{g\in G^{+}_{i,\rho_h(j)}}g(\alpha) -\sum_{g\in G^{-}_{i,\rho_h(j)}}g(\alpha) 
                  =  \sum_{g\in G^{+}_{i,j}}hg(\alpha) -\sum_{g\in G^{-}_{i,j}}hg(\alpha)
=h(m_{i,j});$$
in the second case, we get
$$m_{i\rho_h(j)} = \sum_{g\in G^{+}_{i,\rho_h(j)}}g(\alpha)
-\sum_{g\in G^{-}_{i,\rho_h(j)}}g(\alpha) = \sum_{g\in
  G^{-}_{i,j}}hg(\alpha) -\sum_{g\in
  G^{+}_{i,j}}hg(\alpha)=-h(m_{i,j}).$$ In other words, $h(M_j) = \pm
M_{\rho_h(j)}$, so $G$ permutes the columns of $M$ up to sign.
Secondly, the first column $M_1$ has entries 
$$\sum_{g \in G^{+}_{i,1}}g(\alpha)- \sum_{g \in
  G^{-}_{i,1}}g(\alpha), i = 1, \ldots,n.$$ Since $\alpha$ is a normal
element of the Galois extension $K/F$ with Galois group $G$, and since
$G = \sqcup^n_{i =1}G^{\pm}_{i,j}= \sqcup^n_{i =1}(G^{+}_{i,1} \sqcup
G^{-}_{i,1}) $ is a disjoint union, with all $G^{\pm}_{i,j}$ non-empty (by
the transitivity of the action), this set is $F$-linearly independent.

So Lemma \ref{signdet} applies, and we conclude that $K(y_1, \ldots,
y_n)=K(z_1,\dots,z_n)$; this implies that
$K(x_1,\dots,x_n)^G=F(y_1,\dots,y_n)$.  This finishes the proof in the
transitive case; the proof in the general case follows as in
Theorem~\ref{nonamenonsign}.
\end{proof}

\begin{example}
Let $K= \Q(\rho)$, where $\rho$ is a primitive $5$-th root of unity,
so that $K/\Q$ is Galois, with $\mathrm{Gal}(K/\Q) \cong C_4$.
Take $n=3$, assume $G\leq \mathrm{GL}(3,\Z)$ is generated by $$
\sigma = \begin{bmatrix}
0&-1&0\\
1&0&0\\
0&0&-1
\end{bmatrix}$$ and let $(x_1, x_2, x_3)$ be
indeterminates over $K$, on which $G$ acts as in Definition~\ref{Assumption};
this action is not transitive.
One can also verify that $G$ is isomorphic to $\mathrm{Gal}(K/\Q)$;
$\sigma(\rho)=\rho^2$.  This implies that $\sigma^k(\rho)=\rho^{2^k}$, $k=0,1,2,3$.
In particular, $\sigma^3(\rho)=\rho^{8}=\rho^3$.

We choose $\rho$ as our normal element of the extension $K/\Q$.
(A primitive $p$th root of unity is a normal element for the extension over $\Q$ that it generates.)

In this example, note that $G^+_{1,1}=G^+_{2,2}=\{\mathrm{id}\}$, $G^-_{1,1}=G^-{2,2}=\{\sigma^2\}$, $G^+_{1,2}=G^-_{2,1}=\{\sigma\}$, and $G^-_{1,2}=G^+_{2,1}=\{\sigma^3\}$. Also, $G^+_{3,3}=\langle\sigma^2\rangle$ and $G^-_{3,3}=\sigma G^+_{3,3}=
\langle \sigma,\sigma^3\}$.

With $(z_1,z_2,z_3)$ and $(y_1,y_2,y_3)$
defined as before, the matrix $M^*$ giving the coordinates of
$(1,y_1,y_2,y_3)$ on the basis $(1,z_1,z_2,z_3)$ is
$$M^*= \begin{bmatrix}
1 &0&0&0\\
\rho^3+\rho^4 & \rho -\rho^4 & \rho^2 -\rho^3 & 0\\
\rho^2+\rho^4 & \rho^3 -\rho^2 & \rho -\rho^4 &0 \\
\rho^2+\rho^3 &0 &0 & \rho - \rho^2 -\rho^3+\rho^4
\end{bmatrix}.
$$ Remark that due to the non-transitivity of the action of $G$, the
bottom-right $3 \times 3$ submatrix of $M^*$, while invertible, does
not satisfy the assumptions of Lemma~\ref{signdet} (this matrix is
block diagonal, with blocks corresponding to $K(z_1,z_2)$ and
$K(z_3)$, for which the lemma applies).
\end{example}

\bibliographystyle{plain}
\bibliography{bibliography}

\end{document}